\documentclass{article}
\usepackage{geometry}[margin=1in]
\usepackage{graphicx} 
\usepackage{amssymb, amsmath, color, amsfonts, amsthm,verbatim,enumerate}
\usepackage{graphicx, graphics, epsf, epsfig}

\newcommand{\kissat}{{\scshape Kissat}}
\newcommand{\nauty}{{\scshape Nauty}}
\newtheorem{theorem}{Theorem}

\newtheorem{lemma}[theorem]{Lemma}

\newtheorem{proposition}[theorem]{Proposition}

\newcounter{encoding}
\newtheorem{encoding}[encoding]{Encoding}

\newcommand{\R}{\mathcal{R}}
\title{The exact value of the Ramsey number $R(K_4-e,K_7)$}
\author{William J. Wesley}
\date{\today}

\begin{document}

\maketitle

\section{Introduction}


The \emph{Ramsey number} $R(G_1,G_2)$ is the smallest $n$ such that every graph on $n$ vertices contains a copy of $G_1$ or a copy of $G_2$ in its complement. The most studied cases are when $G_1 = K_s, G_2 = K_t$, and these are called the \emph{classical} Ramsey numbers and  denoted $R(s,t)$. It is extraordinarily difficult to compute the classical Ramsey numbers exactly. The last exact value $R(4,5) = 25$ was found over thirty years ago by McKay and Radziszowski \cite{R45McKayRadziszowksi}. Values as small as $R(3,10),\ R(4,6)$, and $R(5,5)$ remain unknown,  though recent progress has improved the upper bound on $R(5,5)$ from 48 to 46 \cite{R55Le46}. 

While it is difficult to make progress on small Ramsey number bounds for $G_1 = K_s, G_2 = K_t$, there is a wealth of other choices for $G_1$ and $G_2$. In some cases exact formulas are known or conjectured, but many Ramsey numbers---both small values and asymptotics---have yet to be determined. We refer the reader to the most recent edition of Radziszowski's survey for more details \cite{RamseySurvey}.  

We will focus on Ramsey numbers involving complete graphs $K_n$ and complete graphs missing exactly one edge, which we denote $K_n -e$ or $J_n$. Some authors also use the notation $K_n^-$ or $K_{n-1/2}$ for these graphs. These ``nearly classical" Ramsey numbers $R(J_m,J_n)$ and $R(J_m,K_n)$ are somewhat more well-behaved than their classical counterparts, yet are still elusive and worthy of study in their own right. The table of known values for these numbers (see \cite{RamseySurvey}, Table IIIa) has expanded in recent years. Two impressive results yielding large exact values are $R(J_5,J_6) = 37$ and $R(J_5,J_7) = 65$ \cite{BlockCircRamseyGoedVanOver,LidickyPfenderSDPRamsey}. The lower bounds come from special \emph{block circulant} and \emph{strongly regular graphs}, and the upper bounds are obtained by flag algebra tools. Other recent work has produced the exact values $R(J_5,K_5) = 30$ and $R(K_4,J_6) = 30$ by sharpening the upper bounds \cite{Angeltveit_RJ5K5_eq30, JamesKahanRauer_RK4J6_eq30}. 

The smallest unsolved case is $R(J_4,K_7)$. The lower bound $R(J_4,K_7) \ge 28$ follows from the fact that the complement of the strongly regular 27-vertex Schl\"afli graph contains no $J_4$ and no independent set of size 7 (see Figure \ref{FigSchlafli}). In fact, this graph is also the unique lower bound witness for the statement $R(J_4,J_7) = 28$, which was proved by McNamara and Radziszowski in 1991 
\cite{McNamaraRadzJ4J6_J4J7}. However, there are at least 786098 lower bound witness graphs for $R(J_4,K_7)$.
To our knowledge, the first nontrivial upper bound $R(J_4,K_7) \le 31$ was shown by Boza in \cite{BozaJ4K7_LE31} in 2010, and this was later improved to 30 by Boza and Portillo in \cite{BozaPortilloJ4K7_LE30} (see also previous versions of \cite{RamseySurvey}). About a decade later, Lidick\'y and Pfender slashed the upper bound to 29 \cite{LidickyPfenderSDPRamsey}, leaving a gap of 1 between the lower and upper bounds. This paper establishes the exact value for $R(J_4,K_7)$.  

\begin{theorem} \label{ThmMain}
    $R(J_4,K_7) = 28$.
\end{theorem}

It is somewhat surprising that there is no disparity between $R(J_4,J_7)$ and $R(J_4,K_7)$. In fact, we have $R(J_4,J_n) < R(J_4,K_n)$ for $n \le 6$ and $n = 8$. It may also be true that $R(J_5,J_7) = R(J_5,K_7) = 65$. Another recent work of Zhang and Chen gives an interesting connection between $R(J_4,K_n)$ and Ramsey numbers for \emph{wheel graphs} $W_n = K_1 + C_{n-1}$, where $+$ denotes the graph join. Namely, $R(W_7,K_n) \ge 2R(J_4,K_n)-1$, with equality holding for $n \le 4$ and conjectured for all $n$ \cite{ZhangChen_Ramsey_center}. 

\begin{figure} \label{FigSchlafli}
\begin{center}
    \includegraphics[scale = 0.35]{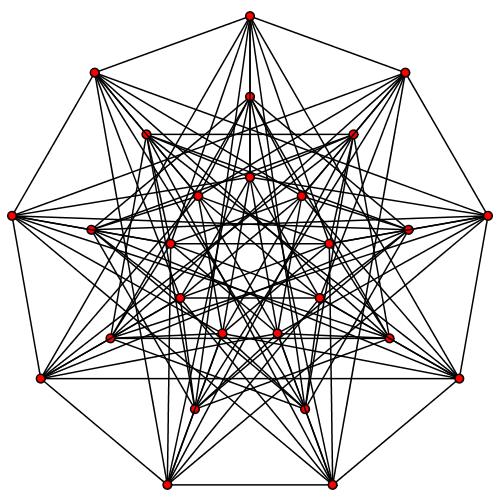}
    \end{center}
    \caption{The complement of the Schl\"afli graph gives the lower bound $R(J_4,K_7) \ge 28$. Credit to House of Graphs for the vertex coordinates of this symmetric drawing \cite{HouseOfGraphs}.}
\end{figure}

Our proof of Theorem \ref{ThmMain} uses a mixture of pen-and-paper and computational methods. In particular, we make liberal use of Boolean satisfiability (SAT) solvers, which have become a staple in computational combinatorics in recent years. SAT solvers are useful for Ramsey-type problems both in finding interesting graphs or colorings and proving that no such graph exists \cite{WJWBooks, WJW_Rado_ISSAC, SchurFive, AhmedZamanBright, BMRS_3ColorSchur, WJW_SmallMulticolor, VDW26, VDW34, R334Equals30}. We also use the \emph{SAT modulo symmetries} (SMS) framework for enumerating graphs and generating certain constraints \cite{SMS}. We give a brief overview of SAT solving in Section \ref{SectionSAT}. The spirit of our proof is similar to those for $R(K_4,J_6)$ and $R(J_5,K_5)$ \cite{Angeltveit_RJ5K5_eq30, JamesKahanRauer_RK4J6_eq30} in that we use known computations and enumerations for smaller Ramsey numbers to rule out the possibility $R(J_4,K_7) > 28$. We describe these prior computations and collect other basic results in Section \ref{SectionNotation}. The main proof of Theorem \ref{ThmMain} is in Section \ref{SectionMainProof}.

Finally, we remark that while a nontrivial amount of computation time was expended for this proof, the total time can be measured in CPU days rather than years. All results were obtained on a personal laptop with an Intel i7-13650HX (2.60 GHz) processor.  

\section{SAT and SMS}\label{SectionSAT}

In this section we give an overview on the SAT solving methods used in this paper. A more comprehensive resource on SAT can be found in \cite{SATHandbook_1stEd}. 


The following encoding is the basis for our SAT computations.
\begin{encoding}\label{EncodingRamseySAT}
	The Ramsey number $R(G_1,G_2)$ is at most $n$ if the formula $F_n(G_1,G_2)$ is unsatisfiable, where 
	$$F_n(G_1,G_2) :=\Bigl(\bigwedge_{H\le K_n, H\cong G_1} \Bigl(\bigvee_{e\in E(H)} \bar x_e\Bigr) \Bigr)\wedge \Bigl(\bigwedge_{H\le K_n, H\cong G_2} \Bigl(\bigvee_{e\in E(H)} x_e\Bigr)\Bigr).$$
	Moreover, if $F_n(G_1,G_2)$ is satisfiable, then $R(G_1,G_2) \ge n+1$.
\end{encoding}
However, it is not enough to simply run a solver with input $F_{28}(J_4,K_7)$. This formula alone is too difficult to be solved by current ``out of the box" solvers in a reasonable amount of time. Indeed, the formula $F_{28}(J_4,K_7)$ is quite large, and it consists of 378 variables and 1306890 clauses. We will need to insert additional clauses to enforce \emph{symmetry breaking} and \emph{cardinality} constraints.  

Symmetry breaking is a powerful---and often necessary---tool in SAT solving. Combinatorial SAT instances often have many symmetries. For instance, the formula $F_n(G_1,G_2)$ is invariant under any permutation of the vertices of $K_n$ (more precisely, a permutation of $V(K_n)$ induces a permutation of the variables $x_e$). A symmetry of a formula $\phi$ induces a group action on the set of truth assignments, and a \emph{symmetry breaking predicate} evaluates to true on at least one element from each orbit. Adding symmetry breaking predicates does not affect the satisfiability of $\phi$, but prevents a solver from considering too many assignments from a single orbit. A more precise description of general symmetry breaking methods for SAT formulas and the associated software {\scshape Shatter} can be found in \cite{Shatter}. 

In this note we require only the following. If $\sigma$ is a permutation of variables that leaves a formula $\phi$ invariant, then we can add constraints $x_1\cdots x_n \le \sigma(x_1)\cdots \sigma(x_n)$, where $x_1\cdots x_n$ and  $\sigma(x_1)\cdots \sigma(x_n)$ are viewed as binary strings with the convention false = 0, true = 1. This ensures that any satisfying assignment will be lexicographically minimal. We can encode a constraint of the form $ x_1\cdots x_n \le y_1 \cdots y_n$ with the following clauses:

\begin{align*}
    &e_0, \\
    & \bar{x}_1 \vee y_1, \\ 
    &\bar{e}_{i-1} \vee \bar{x}_i \vee e_i, \\
    &\bar{e}_{i-1} \vee y_i \vee e_i, \\
    &\bar{e}_{i-1} \vee \bar{x}_{i} \vee \bar{x}_{i+1} \vee y_{i+1},\\
    &\bar{e}_{i-1} \vee y_{i} \vee \bar{x}_{i+1} \vee y_{i+1}, \hspace{10em} \forall i, 1\le i \le n-1. \\
\end{align*}

Note that the variable $e_i$ is set to true if the first $i$ bits of $x$ and $y$ are equal. 

SAT modulo symmetries (SMS) is a framework and software package for generating all nonisomorphic graphs that satisfy a given formula. It uses symmetry breaking methods similar to the lexicographic clauses above, but also some more sophisticated dynamic methods \cite{SMS}. We will use SMS to generate some subgraphs that would necessarily be contained in a $R(J_4,K_7;28)$ graph. 

The formula $F_{28}(J_4,K_7)$ makes no additional assumptions on what a satisfying assignment (graph) looks like. But as we will see in Section \ref{SectionNotation}, there are several degree bounds a hypothetical 28-vertex graph containing no $J_4$ and no $\overline{K_7}$ must satisfy, and we will enforce these with cardinality constraints. Given a set $S$ of literals, let $L_k(S)$ be a formula that is true if and only if at most $k$ of the literals in $S$ are set to true. There are many different ways to encode this constraint (see \cite{WynnSATcardinality} for a survey). All cardinality constraints used in this paper were generated with the SMS software package \cite{SMS}.

\section{Notation and basic results}\label{SectionNotation}

We say that a graph $G$ is a \emph{(Ramsey) $(G_1,G_2;n)$ graph} if it is a graph on $n$ vertices that does not contain a copy of $G_1$ and $\overline{G}$ does not contain a copy of $G_2$. The set of all Ramsey $(G_1,G_2;n)$ graphs is denoted $\R(G_1,G_2;n)$. 

For a given vertex $v$, we let $G_v^+$ be the graph induced on the (open) neighborhood $N(v)$ of $v$, and we let $G_v^-$ be the graph induced on the set $V(G)\setminus (N(v) \cup \{v\})$, the set of non-neighbors of $v$ that are not equal to $v$. For two sets of vertices $A$ and $B$, we let $e(A,B)$ denote the number of edges with one endpoint in $A$ and the other in $B$. We also abuse notation and for graphs $G$ and $H$ we write $e(G,H)$ for $e(V(G),V(H))$. We let $\delta(G)$ and $\Delta(G)$ denote the minimum and maximum degree of $G$. 

We will use the following result, which follows immediately from the definition of Ramsey numbers, frequently.  
\begin{proposition}\label{PropDecomposition}
    Let $G \in \R(J_s,K_t;n)$. If $v \in V(G)$ has degree $d$, then we have 
    \begin{enumerate}[(i)]
    \item $G_v^+ \in \R(J_{s-1},K_t; d)$, 
    \item $G_v^- \in \R(J_s,K_{t-1}; n - d-1)$.
    \end{enumerate}
\end{proposition}

From this it is straightforward to obtain degree bounds on any Ramsey $(J_s,K_t;n)$ graph if the smaller Ramsey numbers are known. 

\begin{proposition}{\label{PropGeneralDegreeBound}}
    Let $G \in \R(J_s,K_t;n)$. Then for all $v \in V(G)$, we have $$n-1 - R(J_s,K_{t-1}) < \deg v < R(J_{s-1},K_t).$$ 
\end{proposition}

With Propositions \ref{PropDecomposition} and \ref{PropGeneralDegreeBound} in mind, it is extremely helpful to have enumerations of the graphs in $\R(J_{s-1},K_{t};d)$ and $\R(J_s,K_{t-1};d)$ for all values of $d$, or when not possible, the values of $d$ that are allowed in Proposition \ref{PropGeneralDegreeBound}. This is a thematic idea in many Ramsey number upper bound proofs. We will use the following enumerations in our main proof: 

\begin{lemma}\label{LemmaEnumerations}
    The following hold: 

\begin{center}
 \normalfont
     \begin{tabular}[t]{|c|c|}
     \hline
         $n$ & $|\R(J_4,K_6;n)|$ \\
         \hline
         19 & 6817238 \\
         20 & 24976 \\
         21 & 0 \\
         \hline
     \end{tabular}\quad
     \begin{tabular}[t]{|c|c|}
     \hline
         $n$ & $|\R(J_4,K_5;n)|$ \\
         \hline
         5 & 21 \\
         6 & 63 \\
         7 & 210 \\
         8 & 897 \\
         9 & 4463 \\
         10 & 23577 \\ 
         11 & 97796 \\
         12 & 180813 \\
         13 & 46510\\
         14 & 856\\
         15 & 13\\
         16 & 0\\
         \hline
     \end{tabular} \quad
      \begin{tabular}[t]{|c|c|}
      \hline
         $n$ & $|\R(K_3,K_6;n)|$ \\
         \hline
         17 & 7 \\
         18 & 0 \\
         \hline
     \end{tabular} \quad
     \end{center}
\end{lemma}

All of these are previously known results. The enumerations of $\R(J_4,K_6;n)$ were done in \cite{ShetlerWurtzRadz} and those for $\R(J_4,K_5;n)$ were done in \cite{Angeltveit_RJ5K5_eq30}. The value $|\R(K_3,K_6;17)|  = 7$ was computed in \cite{RadzKreher_R3k_enums}, and $|\R(K_3,K_6;18)|  = 0$, which is equivalent to the statement $R(K_3,K_6) \le 18$, was originally shown in \cite{R36Kery}. In fact, the values for $|\R(J_4,K_5;n)|$ and $|\R(K_3,K_6;n)|$ are known for all values of $n$, but this is not the case for $|\R(J_4,K_6;n)|$ 

We needed to get our hands on all of these graphs themselves, not merely the counts. For the $R(J_4,K_6;n)$ and $R(K_3,K_6;17)$ graphs we used SAT modulo symmetries (SMS), and for the $R(J_4,K_5;n)$ graphs we used a custom ``vertex extension" program (add one vertex at a time, in all possible ways, then filter out isomorphic copies with the software \nauty\  \cite{Nauty}) to generate all of them. This entire process requires only a few hours of computation time.

We will also use the structure of $R(J_3,K_t;n)$ graphs repeatedly. The following lemma is a folklore or sufficiently ``obvious" result, but we include it and a proof for completeness. 

\begin{lemma}\label{LemMatching}
    Let $G$ be a Ramsey $(J_3,K_t;n)$ graph. Then $G$ is a matching consisting of at least $n-t+1$ edges. Moreover, we have $R(J_3,K_t) = 2t-1$, and in particular, $n < 2t-1$.
\end{lemma}
\begin{proof}
    Notice that $J_3$ is a length 2 path, so each component of a $J_3$-free graph is either a single vertex or a single edge. Let $e$ be the number of edges of $G$. Since $G$ has no independent set of size $t$, we see that $G$ has at most $t-1$ components. Then we must have $e+(n-2e) \le t-1$, so rearranging gives $e \ge n -t+1$ as desired. 

    The lower bound $R(J_3,K_t) \ge 2t-1$ follows from the fact that a matching consisting of $t-1$ edges has no $J_3$ and no $\overline{K_t}$. If there were a $(J_3,K_t;2t-1)$ graph, then from above we have $e \ge t$. However, taking one endpoint from each edge gives an independent set of size $t$, a contradiction.  
\end{proof}

The results of this section are summarized in Figure \ref{FigGeneralStructure}. 
\begin{figure} 

\hspace{5.5em}
    \includegraphics[scale = 0.3, width = \textwidth]{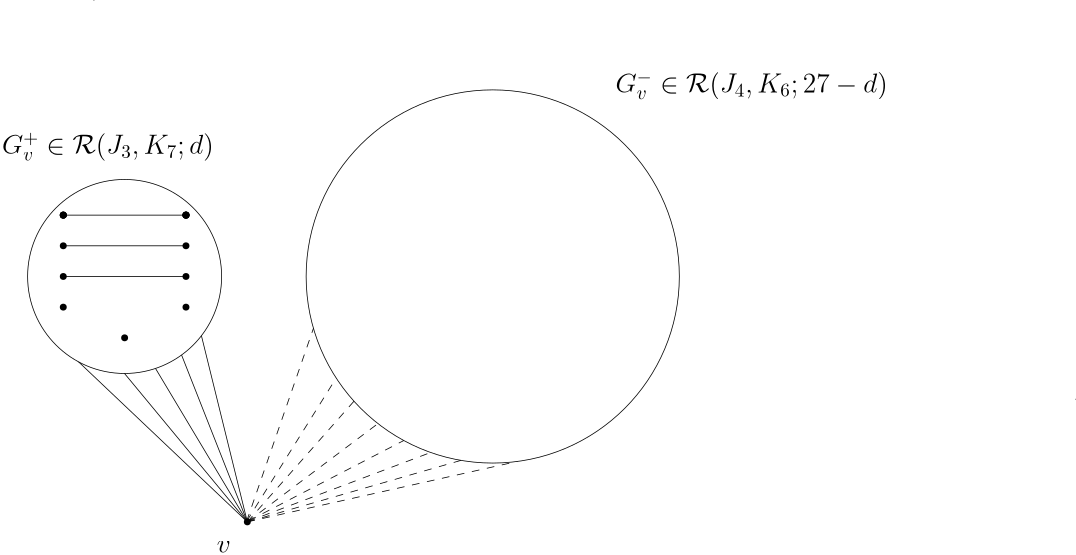}
    \caption{Structure of a hypothetical $R(J_4,K_7;28)$ graph with a degree $d$ vertex $v$. Missing edges are denoted with a dashed line.}
    \label{FigGeneralStructure}
\end{figure}

\section{Proof of Theorem \ref{ThmMain}}\label{SectionMainProof}

The proof of Theorem \ref{ThmMain} is organized by cases based on the minimum degree $\delta$ of a hypothetical Ramsey $(J_4,K_7;28)$ graph $G$. From Proposition \ref{PropGeneralDegreeBound} and using the facts $R(J_3,K_7) = 13$ and $R(J_4,K_6) = 21$, we have that $ 7 \le\deg v \le 12$ for all $v\in V(G)$. The cases for the larger values of $\delta$ are resolved through purely combinatorial arguments, but as $\delta$ gets smaller, we rely on a computational assault. The most complex case was $\delta = 9$, where we split into further subcases and use a combination of both methods. 

\subsection{$\delta = 11,12$}
The first cases we consider are where all vertices have high degree. The following lemma shows this is impossible.
\begin{lemma}\label{LemmaMinDegLe10}
   Suppose $G$ is a Ramsey $(J_4,K_7;28)$ graph. Then $\delta(G) \le 10$.
\end{lemma}
\begin{proof}
    Suppose $\delta(G) \ge 11$. Let $v \in V(G)$. Then $G_v^+$ is a matching consisting of at least 5 edges, each of which forms a triangle with $v$. Let $\{x,y\}$ be one of these edges. Since $G$ is $J_4$-free, observe that $N(x) \cap V(G_v^+) = \{y\}$ and $N(y) \cap V(G_v^+) = \{x\}$. So $x$ and $y$ both have at least 9 neighbors in $G_v^-$. Since $|V(G_v^-)| \le 16$, we see that $x$ and $y$ have a common neighbor in $G_v^-$, contradicting the fact that $G$ is $J_4$-free.
\end{proof}

\subsection{$\delta = 10$}
We now move to the $\delta = 10$ case, which requires computer assistance. We expended a moderate amount of effort trying to ``decomputerize" this section of the proof, but were unsuccessful. This can likely be done, but would require some tedious case analysis; instead, we opted to delegate to a SAT solver. 
\begin{lemma} \label{LemmaMinDegLe9}
   Suppose $G$ is a Ramsey $(J_4,K_7;28)$ graph. Then $\delta(G) \le 9$.
\end{lemma}
\begin{proof}
    From Lemma \ref{LemmaMinDegLe10}, we can assume that $G$ has a degree 10 vertex $v$, and so $|V(G_v^-)| = 17$. Since $G_v^-$ must not contain an independent set of size 6 and $R(K_3,K_6) = 18$, it follows that either $G_v^-$ contains a triangle or $G_v^-$ is one of the 7 graphs in $\R(K_3,K_6;17)$. 

    We used the SAT solver \kissat\ to show that each case is impossible. Each case consists of the clauses in $F_{28}(J_4,K_7)$ together with the following additional clauses, letting $v = 0$ and $K_n$ have vertex set $\{0,\dots, 27\}$:

    \begin{itemize}
        \item Cardinality clauses guaranteeing each vertex has degree at least 10. These were generated with the SMS package. 
        \item The clauses $x_{v,i}$ for $1 \le i \le 10$ and $\bar{x}_{v,i}$ for $11 \le i \le 27$. 

        \item The clauses $x_{1,2}$, $x_{3,4}$, $x_{5,6}$, and $x_{7,8}$. This can be done since by Lemma \ref{LemMatching}, we know $G_v^+$ is a matching consisting of at least 4 edges. 
        
        \item For the 7 Ramsey $(K_3,K_6;17)$ graphs $H$, the literals for the edges in $G_v^-$ were set to true or false so that $G_v^- \cong H$. 
        \item For the case where $G_v^-$ contains a triangle, the clauses $x_{11,12}, x_{11,13}$, and $x_{12,13}$. 
    \end{itemize}

The cases for the 7 Ramsey $(K_3,K_6;17)$ graphs each finished in less than 6 seconds. The case where $G_v^-$ contains a triangle finished in approximately 2 minutes. 
\end{proof}

\subsection{$\delta = 9$}

The $\delta = 9$ case ended up being the least straightforward, though ultimately it did not require the most computation time. If $\delta = 9$, then there is a vertex $v$ with $G_v^- \in \R(J_4,K_6;18)$. However, there is no published census (enumeration) of $\R(J_4,K_6;18)$. Since $|\R(J_4,K_6;19)| = 6817238$, we suspected that it would be difficult simply to compute $\R(J_4,K_6;18)$, let alone rule out each of these graphs as a possibility for $G_v^-$. Instead, we opted to decompose $G_v^- =: H$ further (see Figure \ref{FigIteratedStructure}) and find a large degree vertex $w$ in $H$. Then Proposition \ref{PropDecomposition} implies that $H_w \in \R(J_3,K_6;k)$ and $H_w^- \in R(J_4,K_5;17-k)$ for some $k$. From Lemma \ref{LemmaEnumerations}, we see that the largest sizes of $\R(J_4,K_5;m)$ occur for $m = 11, 12, 13$, or when $k = 4,5,6$. Lemma \ref{LemmaDeg9DeltaGEQ7} shows that we can avoid computing with these cases by showing $H$ has a vertex with degree at least 7 in $H$. 

\begin{figure}
    \centering
    \includegraphics[scale = 0.45]{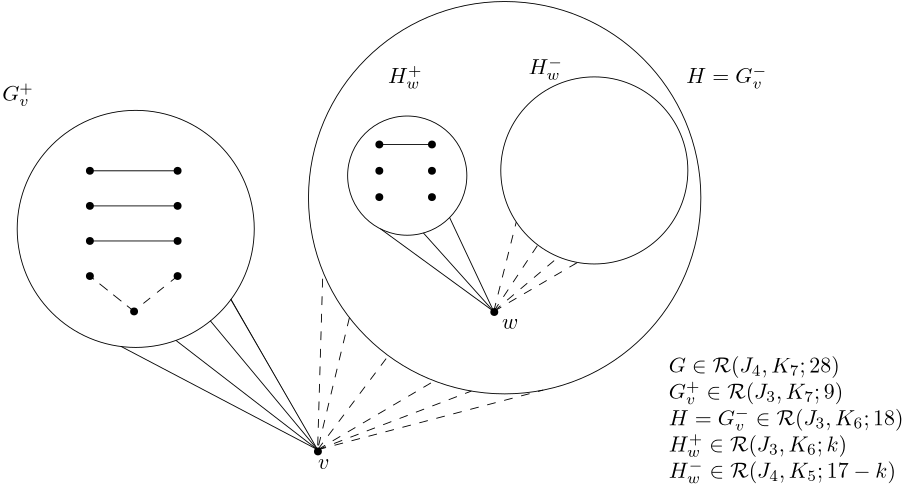}
    \caption{Structure of a hypothetical $R(J_4,K_7;28)$ graph with degree 9 vertex $v$.}
    \label{FigIteratedStructure}
\end{figure}





\begin{lemma}

\label{LemmaDeg9DeltaGEQ7}
    Suppose $G$ is a Ramsey $(J_4,K_7;28)$ graph with $\delta(G) = 9$. Then there is some degree 9 vertex $u$ with $\Delta(G_u^-) \ge 7$.  
\end{lemma}
\begin{proof}
    Let $v$ be a degree 9 vertex. By Proposition \ref{PropDecomposition}, we know that $G_v^-$ cannot have an independent set of size 6. Since $|V(G_v^-)| = 18$ and $R(K_3,K_6) = 18$, we can assume that $G_v^-$ contains a triangle with vertices $\{a,b,c\}$. If $e(G_v^+, \{x\}) < 3$ for some $x \in \{a,b,c\}$, then we have $\deg_{G_v^-} x \ge 7$ and we are done by taking $u := v$. If $e(G_v^+, \{a\}) + e(G_v^+, \{b\}) + e(G_v^+, \{c\}) >9$, then two of $a,b,c$ share a common neighbor $w$ in $G_v^+$, but this is impossible since $G$ is $J_4$-free. Therefore we can assume that $ e(G_v^+, \{x\}) = 3$ for $x \in \{a,b,c\}$. 

    Let $x_i, y_i, z_i$, $i = 1,2,3$, be the respective neighbors of $a,b,c$ in $G_v^+$. Observe that since $G$ is $J_4$-free, the $x_i$ form an independent set, and likewise for the $y_i$ and $z_i$. Suppose that $\deg_{G_v^-} x <7$ for $x \in \{a,b,c\}$. Since $\delta(G) = 9$, we have that each of $a,b,c$ has degree 9 and has exactly 4 neighbors in $V(G_v^-)\setminus \{a,b,c\}$. Let $q_i, r_i, s_i$, $i = 1,2,3,4$, be the neighbors of $a$, $b$, and $c$, respectively, in $G_v^-$.

    Since $G$ has no independent set of size 7, by Lemma \ref{LemMatching} it follows that $G_v^+$ contains at least 3 edges. Without loss of generality, let $a$ be the vertex in $\{a,b,c\}$ that is adjacent to the largest number of endpoints of edges in $G_v^+$. In particular, it must be adjacent to at least 2. We can assume one of the following cases: 
    \begin{enumerate}[(1)]
        \item  $x_1y_1, x_2y_2, x_3y_3$ are edges.
        \item $x_1y_1, x_2y_2, x_3z_1$ are edges.
        \item 
    $x_1y_1, x_2z_1, y_2z_2$ are edges.
    \end{enumerate}
    
    Case 1 is impossible since $\{a,y_1,y_2,y_3,z_1,z_2,z_3\}$ is a size 7 independent set. For Cases 2 and 3, observe that $vx_1y_1$ is a triangle in $G_c^-$, so by a similar argument as above, either we may take $u:= c$, or each of $v$, $x_1$, and $y_1$ is adjacent to exactly 3 vertices in $G_c^+$. Observe that $V(G_c^+) = \{a,b,z_1,z_2,z_3,s_1,s_2,s_3,s_4\}$. We know that $v$ is adjacent to $z_1,z_2,z_3$, that $x_1$ is adjacent to $a$, and $y_1$ is adjacent to $b$. The remaining four vertices are $s_1,s_2,s_3,s_4$, and without loss of generality we may suppose that $s_1$ and $s_2$ are adjacent to $x_1$ and $s_3$ and $s_4$ are adjacent to $y_1$. 
    
    Similarly, in Case 2, using that $vx_3z_1$ is a triangle in $G_b^-$,  we can assume that $r_1$ and $r_2$ are adjacent to $x_3$ and $r_3$ and $r_4$ are adjacent to $z_1$. 

    In Case 3, we can assume $r_1$ and $r_2$ are adjacent to $x_2$ and $r_3$ and $r_4$ are adjacent to $z_1$. Moreover, $q_1$ and $q_2$ are adjacent to $y_2$ and $q_3$ and $q_4$ are adjacent to $z_2$. See Figure \ref{FigCase3} for a picture of these edges.

    For each of Case 2 and Case 3, we added clauses to the formula $F_{28}(J_4,K_7)$ to force the appropriate edges to be present. The solver {\scshape Kissat} determined that these formulas for Case 2 and Case 3 were unsatisfiable in approximately 2 hours and 1 hour, respectively.
\end{proof}

\begin{figure}
\begin{center}\includegraphics[scale = 0.4]{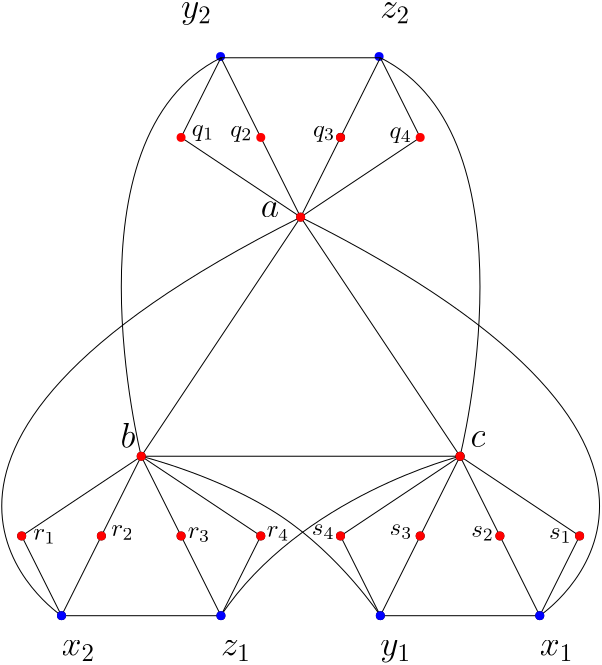}
    \end{center}
\caption{Some of the edges set in Case 3 of Lemma \ref{LemmaDeg9DeltaGEQ7}. Vertices colored blue are adjacent to $v$, and vertices colored red are nonadjacent to $v$.}
\label{FigCase3}
\end{figure}

We can now eliminate the case $\delta = 9$. 

\begin{lemma} \label{LemmaMinDegLe8}
    Suppose $G$ is a Ramsey $(J_4,K_7;28)$ graph. Then $\delta(G) \le 8$. 
\end{lemma}

\begin{proof}
From Lemma \ref{LemmaMinDegLe9}, we can assume $\delta(G) = 9$. Lemma \ref{LemmaDeg9DeltaGEQ7} says that we can find a degree 9 vertex $v$ and a $w$ with $\deg_{G_v^-}w \ge 7$. Set $H = G_v^-$. Then by Proposition \ref{PropDecomposition} and Lemma \ref{LemMatching}, we have that $H_w^+ \in \R(J_3,K_6;k)$ and $H_w^- \in \R(J_4,K_5;17-k)$ for some $k \in [7,10]$. Let $K_{28}$ have vertex set $\{0,\dots,27\}$, and set $v = 0, w = 10$. For each $k \in [7,10]$ and each graph $Q \in \R(J_4,K_5;17-k)$, we constructed a formula $\phi_Q$ that consists of the clauses of $F_{28}(J_4,K_7)$ together with the following additional clauses:

\begin{itemize}
\item Cardinality clauses guaranteeing each vertex has degree at least 9. These were generated with the SMS package.
    \item $x_{0,i}$ for $1 \le i \le 9$, 
    \item $\bar{x}_{0,i}$ for $10 \le i \le 27$,
    \item $x_{10,10+i}$ for $1\le i \le k$,
    \item $\bar x_{10,10+i}$ for $k+1\le i \le 17$,
    \item The literals for the edges in $H_w^-$ were set to true or false so that $H_w^- \cong Q$.
    \item The clauses $x_{1,2}$, $x_{3,4},$ $x_{5,6}$, $\bar{x}_{7,9}$, and $\bar{x}_{8,9}$. We can do this since Lemma \ref{LemMatching} guarantees 3 edges in $G_v^+$. The two negative unit clauses force the last possible edge in $G_v^+$ to be $\{7,8\}$. 
    \item $x_{9+2i,10+2i}$ for $1\le i \le k-5$ since Lemma \ref{LemMatching} guarantees at least $k-5$ edges in $H_w^+$. 
\end{itemize}



We ran all of these formulas $\phi_Q$ in \kissat\ with a time limit of 5 minutes each. All were shown to be unsatisfiable except 243 formulas with $k=7$ and 6 formulas with $k=8$, which did not finish in time.

First, we ran the 6 hard $k=8$ formulas with an increased time limit, and these all finished in times ranging from 314 seconds (slightly over the previous time limit) to 12581 seconds (roughly 3.5 hours). We were concerned that the 243 hard $k=7$ formulas would take too much time, so at this stage we turned to additional symmetry breaking. 

Observe that we can permute the three edges in $G_v^+$ and swap the endpoints of any individual edge without affecting the satisfiability of $\phi_Q$.
Therefore we can add the following symmetry breaking clauses to $\phi_Q$:

\begin{itemize}
    \item $x_{1,10}\cdots x_{1,27} \le x_{2,10}\cdots x_{2,27}$,
   \item  $x_{3,10}\cdots x_{3,27} \le x_{4,10}\cdots x_{4,27} $, 
   \item  $x_{5,10}\cdots x_{5,27} \le x_{6,10}\cdots x_{6,27} $,
   \item $x_{1,10}\cdots x_{1,27} \le x_{3,10}\cdots x_{3,27}$, 
   \item $x_{3,10}\cdots x_{3,27} \le x_{5,10}\cdots x_{5,27}$. 
\end{itemize}

With these additional clauses added, the 243 hard instances took an average of 229 seconds to solve, with a minimum of 19.6 seconds and maximum of 3469 seconds. The total time was 55580 seconds, or about 15.4 hours. 

\end{proof}




\subsection{$\delta$ = 8}


If $G$ has a degree 8 vertex $v$, then we have $G_v^- \in\R(J_4,K_6;19)$. From Lemma \ref{LemmaEnumerations} and \cite{ShetlerWurtzRadz} we know that $|\R(J_4,K_6;19)| = 6817238$. A formula that sets the edges of $G_v^-$ to be a particular one of these Ramsey graphs can be shown to be UNSAT by {\kissat} in a few seconds. It would not be unreasonable to brute force all 6817238 cases in this way, especially with parallelization, but we wanted to get a modest improvement on the total computation time. 

First, we introduced symmetry breaking clauses as follows. Since $\delta = 8$, we have that $G_v^+$ contains at least two edges. Similarly to the $\delta = 9$ case, we can permute the two edges and swap the endpoints of each individual edge. Therefore we can assume the following: 

\begin{itemize}
    \item $x_{1,9}x_{1,10}\cdots x_{1,27} \le x_{2,9}x_{2,10}\cdots x_{2,27}$, 
   \item  $x_{3,9}x_{3,10}\cdots x_{3,27} \le x_{4,9}x_{4,10}\cdots x_{4,27} $,
   \item $x_{1,9}x_{1,10}\cdots x_{1,27} \le x_{3,9}x_{3,10}\cdots x_{3,27}. $ \\
\end{itemize}

We observed that many of the graphs in $\R(J_4,K_6;19)$ are similar to one another, and therefore should produce similar proofs of unsatisfiability. In other words, it is inefficient for the solver to redo similar computations (as well as preprocessing and other procedures) individually for over six million graphs. To circumvent this, we want to batch similar graphs together and have the solver rule them out in a single blow. For a family of graphs $\mathcal F$, we can add the following clauses: 

\begin{align*}
    &\bar s_H \vee x_e,  &&\forall H \in \mathcal F, \forall e \in E(H)\\
    &\bar s_H \vee \bar x_e,  &&\forall H \in \mathcal F, \forall e \in \binom{V(H)}{2} \setminus E(H)\\
    &\bigvee_{H \in \mathcal F} s_H.
\end{align*}

The variables $s_H$ are selectors that ``turn on" the assignment of the edges and non-edges for $H$. The clause $\bigvee_{H \in \mathcal F} s_H$ guarantees that one of the graphs $H$ is actually assigned. 

We ran 6817 batches with  1000 $(J_4,K_6;19)$ graphs at a time using this method (and one final batch with the remaining 238). The graphs were ordered lexicographically by their strings in the graph6 format (see \cite{McKayRamseyGraphs}). The solve times for the size 1000 batches ranged from 6.7 to 2503 seconds with a mean solve time of 43.5 seconds and total solve time of 296521 seconds, or about 82.4 hours. The final batch solved in 7.2 seconds. This ended up being one of the most computationally intensive parts of the proof. It could certainly be optimized further, but we thought that four days of time was reasonable. 

The previous discussion together with Lemma \ref{LemmaMinDegLe8} proves the following. 

\begin{lemma} \label{LemmaMinDegLe7}
    Suppose $G$ is a Ramsey $(J_4,K_7;28)$ graph. Then $\delta(G) \le 7$.
\end{lemma}



\subsection{$\delta = 7$}

This case ends up being simpler than the $\delta = 8$ case because $|\R(J_4,K_6;20)|$ is much smaller than $|\R(J_4,K_6;19)|$.

\begin{lemma}\label{LemmaMinDegGt7}
    Suppose $G$ is a Ramsey $(J_4,K_7;28)$ graph. Then $\delta(G) > 7$.
\end{lemma}
\begin{proof}
    Suppose $G$ has a degree 7 vertex $v$. Then $G_v^-$ is one of the 24976 graphs in  $\R(J_4,K_6;20)$. In a similar way as in previous cases, for each of these graphs $H$, we constructed a formula $\phi_H$ where the literals for the edges in $G_v^-$ are set so that $G_v^-$ is isomorphic to $H$. Each formula $\phi_H$ was determined to be unsatisfiable by \kissat. 
\end{proof}
Here each instance was easy and took only a few seconds to solve. This step was actually completed prior to the $\delta = 8$ case, and it certainly could be made faster with the techniques done there. 

Lemmas \ref{LemmaMinDegLe7} and Lemmas \ref{LemmaMinDegGt7} imply that a Ramsey $(J_4,K_7;28)$ graph does not exist. This concludes the proof of Theorem \ref{ThmMain}. 

\section{Concluding Remarks}

Our goal for this project was to prove the upper bound $R(J_4,K_7) \le 28$, but we did not compute a census of $\R(J_4,K_7;27)$, and this set contains at least 786098 graphs, according to an SMS computation. There may be more, and perhaps many more. It would be interesting to determine the exact number and explore whether any of them have special properties like the Schl\"afli complement.

The next unresolved case in Table IIIa of \cite{RamseySurvey} is $R(J_4,J_8)$. The current bounds stand at $$30 \le R(J_4,J_8) \le 32.$$ 
We believe that there is a good chance progress can be made on the upper bound using similar methods, but we leave this to future work. 

\section*{AI Declaration}
The author used some AI models for proofreading and minor code optimization. AI did not play any other role in the generation of the text or mathematical content for this paper.

\bibliographystyle{plain}

\end{document}